\documentclass[smallextended,envcountsect,envcountsame]{svjour3}

\usepackage[latin1]{inputenc}
\usepackage[cmex10]{amsmath}
\usepackage{hyperref}
\usepackage{amsfonts}
\usepackage{xcolor}
\usepackage{amssymb}
\usepackage{graphics}
\usepackage[ruled,vlined]{algorithm2e}
\usepackage[top=1.03in, bottom=1.03in, left=1.03in, right=1.03in]{geometry}
\usepackage{amssymb,amsmath,amscd,epsf,latexsym,verbatim,graphicx,amsfonts}
\usepackage{authblk}

\newcommand{\mcb}{\mathcal{B}}

\newcommand{\mcm}{\mathcal{M}}

\newcommand{\mcn}{\mathcal{N}}
\newcommand{\mcd}{\mathcal{D}}

\def\Znum{{\mathbb{Z}}} 
\newcommand{\br}[1]{\langle {#1} \rangle }

\author{K. T. Arasu \and
Daniel M. Gordon \and
Yiran Zhang
\thanks{REU research supported by an NSF grant, and presented at the Young Mathematicians' Conference at OSU August 2015}
}

\institute{K. T. Arasu 
\at Riverside Research, 2640 Hibiscus Way, Beavercreek, OH 45431,\\\email{karasu@riversideresearch.org}
\and
Daniel M. Gordon 
\at Center for Communications Research, 4320 Westerra Court, San Diego, CA 92121,\\\email{gordon@ccrwest.org}
\and
Yiran Zhang
\at Loyola University, Chicago, IL 60660,\\\email{yzhang45@luc.edu}
}

\title{New Nonexistence Results on Circulant Weighing Matrices}

\date{November 2, 2020}

\begin{document}

\maketitle

\begin{abstract}
A weighing matrix $W = (w_{i,j})$
is a square matrix of
order $n$ and entries $w_{i,j}$ in $\{0, \pm 1\}$ such that
$WW^T=kI_n$.  In his thesis, Strassler gave a table of
existence results for circulant weighing matrices with $n \leq 200$ and $k \leq 100$.

In the latest version of Strassler's table given by Tan,
there are 34 open cases remaining.  In this paper we give nonexistence
proofs for 12 of these cases, 
report on preliminary searches outside Strassler's table,
and characterize the known proper circulant weighing matrices.


\end{abstract}

\section{Introduction}


A {\em weighing matrix} $W=W(n,k)$ with weight $k$ is a square matrix of
order $n$ with entries $w_{i,j}$ in \{-1, 0, +1\} such that $WW^T=kI_n$
where $W^T$ is the transpose of $W$ and $I_n$ is the $n\times n$
identity matrix. 

A circulant weighing matrix $C = CW(n,k)$ is a weighing
matrix in which every row except for the first is a right cyclic shift
of the previous row.  Let $P$ be the set of locations with a $+1$ in
the first row, and $N$ be the locations with a $-1$.  Then
$|P|+|N|=k$.






The following facts are well known:

\begin{enumerate}

\item $k=s^2$ for some positive integer $s$

\item $|P|=\frac{s^2+s}{2}$

\item $|N|=\frac{s^2-s}{2}$

\end{enumerate}

$|P|$ and $|N|$ are chosen by convention, since $-C$ is also a
circulant weighing matrix.
For further information on
weighing matrices, we refer the reader to \cite{arasuhollon},
\cite{arasuseberry2}, \cite{arasuseberry}.

In his 1997 thesis \cite{strassler}, Strassler gave a table of
known results on such matrices with $n \leq 200$ and $k \leq 100$.
Over the years, many open cases in his table have been resolved.
In Tan's 2018 version of the table \cite{tan}, there are 34 open cases remaining.
In this paper, we will show that no $CW(n,k)$ exists for twelve of those cases.

\section{Group Rings and Multipliers}

It is convenient to think of circulant weighing matrices
$CW(n,k)$ as elements of a group ring.
Let $R$ be a commutative ring with identity $i_R$ and
$G$ be a finite multiplicatively written group of order $n$. Let
$R[G]=\{\sum_{g\in G}a_gg\mid a_g\in R\}$ denote the group ring of $G$
over $R$.

\begin{definition}
For an integer $t$ and $A = \sum_{g \in G} a_g g$, define $A^{(t)}  =\sum_{g \in G} a_g g^t$.
\end{definition}

In this paper, we will be working with $\Znum[\Znum_n]$, the group ring of
the cyclic group $\Znum_n$ of order $n$ over  $\Znum$, the
ring of integers. 
A $CW(n,k)$ is an element $A$ of $\Znum[\Znum_n]$ with all
coefficients in $\{0,\pm 1\}$ such that 
\begin{equation}\label{eq:cw}
A  A^{(-1)}=k.
\end{equation}
If the coefficients of $A$ are in $\{0,\pm 1, \pm 2, \ldots, \pm m\}$,
then we will call it an {\em integer circulant weighing matrix},
denoted $ICW_m(n,k)$.

Representing elements of $\Znum_n$ as 
$\{1,X,X^2,\ldots,X^{n-1}\}$ modulo $(X^n-1)$, 
we may think of $CW(n,k)$ as a polynomial in
$\Znum[X]/(X^n-1)$.  For example, 
a circulant weighing matrix $CW(7,4)$
$$ 
\left[
\begin{array}{ccccccc}
 - & + & + & 0 & + & 0 & 0 \\
 0 & - & + & + & 0 & + & 0 \\
 0 & 0 & - & + & + & 0 & + \\
 + & 0 & 0 & - & + & + & 0 \\
 0 & + & 0 & 0 & - & + & + \\
 + & 0 & + & 0 & 0 & - & + \\
 + & + & 0 & + & 0 & 0 & - \\ 
\end{array} 
\right]
$$
is equivalent to
$$
A(X) = -1 +X + X^2 + X^4.
$$
The group ring element is given by the first row of the matrix, and
(\ref{eq:cw}) applies both the the matrix and group ring element.

We will generally  leave the $\bmod (X^n-1)$ implicit.
For any integer $s$, $X^{s} A(X)$ is an equivalent CW, a cyclic shift of $A$.
For an integer $t$, $A^{(t)} = A(X^t)$ denotes the image
of $A$ under the group homomorphism $x\rightarrow x^t$, extended
linearly to all of $\Znum[\Znum_n]$.
If $\gcd(t,n)=1$, then this map is an automorphism.
If $\gcd(t,n)=d$, then $A^{(t)}$ is an $ICW_d(n/d,k)$.

A prime $p$ is called {\em self-conjugate} modulo $n$ if there is an
integer $i$ with 
$$
p^i \equiv -1 \bmod {v(n)},
$$
where $v(n)$ is the largest divisor of $n$ relatively prime to $p$.
The following result of Lander, given in this form for group rings in \cite{jungnickel}, will be
used below.

\begin{theorem}\label{thm:turyn}
For an abelian group $G$ of order $n$, if $A \in \Znum[G]$ satisfies
$$
A A^{(-1)} \equiv 0 \bmod p^{2a},
$$
for a positive integer $a$ and prime $p$, and $p$ is self-conjugate mod $n$
then,
$$
A  \equiv 0 \bmod p^{a}.
$$
\end{theorem}


Next, we will discuss multipliers. 

\begin{definition}
Let $G$ be a finite abelian group
of order $n$ and $D$ be a subset of $G$. Let $t$ be an integer
relatively prime to $n$. If $D^{(t)}=Dg$ for some $g$ in $G$, then $t$
is called a multiplier of $D$.
\end{definition}




The following theorem is well-known; see, for example, \cite{arasuseberry2}:

\begin{theorem}\label{thm:mult}
  Let $A$ be a $CW(n,k)$, where $k=p^{2r}$
  is a prime power, and $\gcd(n,k)=1$.  Then $p$ is a multiplier of
  $A$.  Furthermore, $p$ fixes some translate of $A$.
\end{theorem}

We will frequently use this theorem.  When a $CW(n,k)$ has a
multiplier $p$, then some translate of it is fixed by the group
generated by $p \bmod n$, and so $P$ and $N$ must both be unions of
orbits of $\Znum_n$ under the action of multiplying by $p$.  
For example, $2$ is a multiplier of $CW(7,4)$, and the presentation
given above is fixed by it:
$$
A^{(2)} = A(X^2) = -1 + X^2+ X^4+ X^8 \equiv A \pmod{X^7-1}.
$$
The orbits of $2 \bmod 7$ are $\{0\}$, $\{1,2,4\}$ and $\{3,5,6\}$,
so the only possibilities are $N=\{0\}$ and $P$ being one of the
other orbits, both of which give (equivalent) $CW(7,4)$s.
Needing
to exhaust unions of orbits instead of arbitrary subsets will often
transform an infeasible search into a feasible one, and allow us to
handle the cases in this paper by hand.


Let $n=dm$, with $d,m>1$.  We may reduce $A = \sum_{i=0}^{n-1} a_i X^i$ modulo $X^m-1$ to get
$$
B = \sum_{i=0}^{m-1} \left( \sum_{j=0}^{d-1} a_{i+jm}
\right)X^{i}  = \sum_{i=0}^{m-1} b_i X^{i}.
$$

The $b_i$'s are called the {\em intersection numbers}.
They have been extensively used for studying the existence of difference sets,
circulant weighing matrices (\cite{strassler} called it {\em folding}), and 
supplementary difference sets \cite{dk2015}.

\begin{lemma}\label{lem:subgp}
For an $ICW(n,k)$ circulant weighing matrix $A$ as above,
\begin{align}
\sum_{i=0}^{m-1} b_i & =  s,\label{eq:s}\\ 
\sum_{i=0}^{m-1} b_i^2 & =  s^2 = k,\label{eq:k}
\end{align}
where $|b_i| \leq d$.
\end{lemma}

Note that if $d \leq s$, then for any $0 \leq i  < m$, these equations
have a
trivial solution $b_i=s$ and $b_j=0$ for $i \neq j$.  If
Theorem~\ref{thm:turyn} applies with $p^a = s$, then the trivial
solutions are the only ones.

$B(X)$ is an $ICW_d(m,w)$.  If
$A(X)$ is equivalent to $B(X^d)$ for any $d>1$, then $A$ is called a {\em multiple}
of $B(X)$.  
If $A(X)$ is not a multiple of any $CW$, then it is called {\em proper}.

\bigskip



In this paper, we will consider possible
circulant weighing matrices in $\Znum_{n} = \Znum_{d} \times \Znum_{m}$,
where $d$ and $m$ are relatively prime.  If $p$ is a multiplier for a
$CW(n,k)$ matrix $A$, then we may assume that a translate of $A$ is
fixed by the group $\br{p} \pmod n$, and so $P$ and $N$ must each be
the union of orbits of the multiplier group.  This applies to the folded versions in
$\Znum_{d}$ and $\Znum_{m}$ as well, so we may use Lemma~\ref{lem:subgp}
to get information about what orbits are in $P$ and $N$ in the two
subgroups, and from that limit the possibilities for orbits in the
full group.

For any of the cases in Strassler's table, Equations~(\ref{eq:s}) and
(\ref{eq:k}) will be small enough that we can solve them either by
hand or with a short computer exhaust.  The bulk of each proof will
show that none of those pairs of solutions corresponds to a
circulant weighing matrix.

Since this method will be used repeatedly, we will formalize it here.
Let $\sigma$ and $\tau$
be projections from $\Znum_{n}$ to $\Znum_{d}$ and $\Znum_{m}$, respectively.
Let the orbits of $\Znum_{n}$ be
$$
\mcn_1 , \mcn_2 , \ldots \mcn_w,
$$ 
the orbits of $\Znum_{d}$ be
$$
\mcd_1 , \mcd_2 , \ldots \mcd_u,
$$ 
and the orbits of $\Znum_{m}$ be
$$
\mcm_1 , \mcm_2 , \ldots \mcm_v.
$$ 
$\mcb_{ij} = \{\mcb_{ij}^1,\mcb_{ij}^2,\ldots,\mcb_{ij}^l\}$ will denote the set of orbits
$\mcn$ which map to orbits $\mcd_i$ and
$\mcm_j$ under $\sigma$ and $\tau$.
This information may be represented as a matrix shown in Table~\ref{tab:generic},
where the row and
column sums ${\bf r}=(r_1,\ldots,r_u)$ and ${\bf c}=(c_1,\ldots,c_v)$ are the sum of the orders of the orbits in
that line in $P$ minus the sums of the orders in $N$.  
The intersection numbers of Lemma~\ref{lem:subgp} are $r_i/|\mcd_i|$
and $c_j/|\mcm|$.

\begin{table}
\begin{center}
\begin{tabular}{c||c|c|c|c||c}
\multicolumn{4}{c}{\ \ \ \ \ \ \ \ \ \ $\Znum_{m}$}\\
$\Znum_{d}$ & $\mcm_1$ & $\mcm_2$ & $\ldots$ & $\mcm_v$ &  \\ \hline  \hline 
$\mcd_1$      &        $\mcb_{1,1}$           &       $\mcb_{1,2}$           &           &            $\mcb_{1,v}$       & $r_1$ \\ \hline 
$\mcd_2$      &  $\mcb_{2,1}$                 &       $\mcb_{2,2}$           &            &            $\mcb_{2,v}$       & $r_2$ \\ \hline 
$\vdots$      &                  &             &      $\ddots$         &               &  \\ \hline
$\mcd_u$      &           $\mcb_{u,1}$        &   $\mcb_{u,2}$               &             &         $\mcb_{u,v}$          & $r_u$ \\ \hline  \hline 
 & $c_1$ &  $c_2$ &  & $c_v$ &  
\end{tabular}
\caption{Orbit table for $CW(n,k)$}
\label{tab:generic}
\end{center}
\end{table}

To illustrate the method that will be used throughout this paper,
consider a $CW(63,16)$.  By Theorem~\ref{thm:mult} 2 is a multiplier, and
the orbits of $\Znum_{63}$ are shown in Table~\ref{tab:63_4},
where each orbit is represented by its generator in the appropriate
group, and the subscript gives the size of the orbit.
A $CW(63,16)$ is given in Table~\ref{tab:cw}
where the sets in $P$ are in bold.  Note that the intersection numbers
$(4,0,0)$ and $(1,2,-1)$ satisfy Lemma~\ref{lem:subgp}.  Also, 
by Theorem~\ref{thm:turyn} with $n=9$ and $p = 2$ we have
the intersection numbers mod 9 must be trivial.
There is one other inequivalent
solution, which has the same intersection numbers.

\begin{table}
\begin{center}
\begin{tabular}{c||c|c|c||c}
 \multicolumn{4}{c}{\ \ \ $\Znum_7$} \\ 
$\Znum_9$ & $\br{ 0 }_{1}$  & $\br{ 1 }_{3}$  & $\br{ 3 }_{3}$ \\ \hline \hline
$\br{ 0 }_{1}$  & $\br{ 0 }_{1}$  & $\br{ 9 }_{3}$  & $\br{ 27 }_{3}$  &  \\ \hline
$\br{ 1 }_{6}$  & $\br{ 7 }_{6}$  & $\br{ 1 }_{6}$ $\br{ 11 }_{6}$ $\br{ 23 }_{6}$  & $\br{ 5 }_{6}$ $\br{ 13 }_{6}$ $\br{ 31 }_{6}$  &  \\ \hline
$\br{ 3 }_{2}$  & $\br{ 21 }_{2}$  & $\br{ 15 }_{6}$  & $\br{ 3 }_{6}$  &  \\ \hline
\end{tabular}
\end{center}
\label{tab:63_4}
  \caption{$CW(63,4^2)$ Orbits}
\end{table}

\begin{table}
\begin{center}
\begin{tabular}{c||c|c|c||c}
 \multicolumn{4}{c}{\ \ \ $\Znum_7$} \\ 
$\Znum_9$ & $\br{ 0 }_{1}$  & $\br{ 1 }_{3}$  & $\br{ 3 }_{3}$ \\ \hline \hline
$\br{ 0 }_{1}$  & $\mathbf{\br{ 0 }_{1}}$  &  & $\mathbf{\br{ 27 }_{3}}$  & 4  \\ \hline
$\br{ 1 }_{6}$  &  & $\mathbf{\br{ 11 }_{6}}$  & { $\br{ 31 }_{6}$ } & 0  \\ \hline
$\br{ 3 }_{2}$  &  &  &  & 0  \\ \hline
 &  1 &  $3\cdot (2)$ &  $3 \cdot (-1)$ & 
\end{tabular}
\end{center}
\label{tab:cw}
  \caption{Solution representing a $CW(63,4^2)$}
\end{table}

In the next section we will use this framework to demonstrate that
various $CW(n,k)$ do not exist.  For small cases this can be done by
hand.  
A computer search dealing with larger cases will be discussed in Section~\ref{sec:strassler}.

\section{Nonexistence Results}

In this section, we present proofs of nonexistence for several of the
open cases in Strassler's table.  All these proofs may be done by
hand, without computer assistance. In later sections, we will look at
more difficult parameters, where more substantial computation is needed.


\begin{proposition} \label{prop:110-81}
A $CW(110,81)$ does not exist.  
\end{proposition}

\begin{proof}
Suppose a $CW(110,81)$ exists.  
Then $|P|=45,|N|=36$ and 3 is its multiplier.  Note
that $\Znum_{110}=\Znum_{11}\times\Znum_{10}$.  
The $\Znum_{11}$ orbits
under the multiplier action $x\rightarrow3x$ are: \[ \{0\} \] \[
\{1,3,9,5,4\} \] \[ \{2,6,7,10,8\} \]

The $\Znum_{10}$ orbits under the multiplier action $x\rightarrow3x$
are: \[ \{0\} \] \[ \{5\} \] \[ \{1,3,9,7\} \] \[
\{2,6,8,4\} \]


Since $3$ is self-conjugate mod 10, by Theorem~\ref{thm:turyn} the
intersection numbers mod 10 are trivial.
Without loss of generality we may take the first row sum to be 9, and
the others zero (there is no way to get the last two rows to sum to
zero, and if the second row sum was 9 we can shift each element of $P$
and $N$ by 55).
The only way for the first row to sum to 9 is for
$\br{ 0 }_1$ to be in $N$, and 
$\br{ 20 }_5$ and $\br{ 10 }_5$ to be in $P$.

Applying Lemma~\ref{lem:subgp}, we get equations
\begin{equation}\label{eq:110a}
y_1 + 5(y_2+y_3)  =  9 
\end{equation}
and
\begin{equation}\label{eq:110b}
y_1^2 + 5(y_2^2 + y_3^2)  =  81
\end{equation}
for $\Znum_{11}$.
Since for each row the size of the first column orbit
is different from the other two, the other row sums being zero means that 
their first column orbit cannot be in $P$ or $N$, 
so we must have $y_1 = -1$.  But
equations (\ref{eq:110a}) and (\ref{eq:110b}) have no integer
solutions with $y_1 =-1$.

\begin{table}
\begin{center}
\begin{tabular}{c||c|c|c||c}
 \multicolumn{4}{c}{\ \ \ \ \ $\Znum_{11}$} \\
$\Znum_{10}$    & $\br{ 0 }_1$   & $\br{ 1 }_5$ & $\br{ 2 }_5$  &  \\ \hline \hline
$\br{ 0 }_1$  & $\br{ 0 }_1$ & $\br{ 20 }_5$ & $\br{ 10 }_5$ & $9$ \\ \hline
$\br{ 5 }_1$  & $\br{ 55 }_1$ & $\br{ 5 }_5$ & $\br{ 35 }_5$ & $0$ \\ \hline
$\br{ 1 }_4$  & $\br{ 11 }_4$ & $\br{ 3 }_{20}$ & $\br{ 7 }_{20}$  & $0$ \\ \hline
$\br{ 2 }_4$  & $\br{ 22 }_4$ & $\br{ 4 }_{20}$ & $\br{ 2 }_{20}$ & $0$\\ \hline \hline
 & $y_1$  & $5y_2$  & $5y_3$ & \\
\end{tabular}
\caption{Orbit information for $CW(110,81)$}
\label{tab:110}
\end{center}
\end{table}
  
\end{proof}

\bigskip


\begin{proposition} 
Suppose $m$ is an integer for which $\gcd(33,m)=1$, and $3$ is
self-conjugate modulo $m$.  Then no $CW(11\cdot m,81)$ exists.
\end{proposition}

\begin{proof}

For any $n = 11 \cdot m$, we may make a table of orbits similar to
Table~\ref{tab:110}.  Since $\gcd(3,m)=1$, 3 is a multiplier.
The orbits mod $11$ will be the same, and since
3 is self-conjugate mod $m$,, by Theorem~\ref{thm:turyn}
the intersection numbers mod $m$ must be
trivial.  As before the first row sum must be 9,
so that again the $\br{0}_1$ orbit must be in $N$. 
All the other
row sums are $0$, and since each row has orbits of size 
$o$, $5o$ and $5o$, that means that the orbit in the first column
cannot be in $P$ or $N$.  Therefore $y_1$ in 
equations (\ref{eq:110a}) and (\ref{eq:110b}) would need to be $-1$,
and those equations still have no such integer solutions.

\end{proof}

This rules out many such parameters, one of which is in Strassler's
table and was open:

\begin{corollary}
A $CW(154,81)$ does not exist.  
\end{corollary}

\bigskip

The same method, with different orbits, may be used for other parameters.


\begin{proposition} A $CW(130,81)$ does not exist.  \end{proposition}

\begin{proof}

The orbit information is given in Table~\ref{tab:130}.  While $3$ is
self-conjugate modulo $10$, the orbit structure is different, so
the argument is not quite as straightforward.

Without loss of generality, by Theorem~\ref{thm:turyn} we take the first row sum to be 9,
so three of the four 
3-orbits must be in $P$. 
The other row sums are 0, so the first column can never be included,
and the other four columns for each row must have the same number of
orbits in $N$ and $P$.  
No orbits can come from the second row, since $N$ has order 
$36 \equiv 0 \pmod 12$, and all the other orbits are 12-orbits.

The $\Znum_{13}$ equations from Lemma~\ref{lem:subgp} are:
\begin{equation}\label{eq:130a}
y_{1} + 3 y_{2} + 3 y_{3} + 3 y_{4} + 3 y_{5} = 9
\end{equation}
and
\begin{equation}\label{eq:130b}
y_{1}^2 + 3 y_{2}^2 + 3 y_{3}^2 + 3 y_{4}^2 + 3 y_{5}^2  =  81.
\end{equation}

The only solutions with $y_1=0$ are permutations of $(0,3,3,-3,0)$.
But there is no way to get a column sum of $-3$ with zero or one
3-orbits from the first row in $P$ and some number of 12-orbits in $P$
or $N$.

\begin{table}
\begin{center}
\begin{tabular}{c||c|c|c|c|c||c}
\multicolumn{4}{c}{\ \ \ \ \ \ \ \ \ \ $\Znum_{13}$}\\
$\Znum_{10}$ & $\br{ 0 }_{1}$ & $\br{ 1 }_{3}$ & $\br{ 2 }_{3}$ & $\br{ 4 }_{3}$ & $\br{ 7 }_{3}$ &  \\ \hline  \hline 
$\br{ 0 }_{1}$ & $\br{ 0 }_{1}$ & $\br{ 40 }_{3}$ & $\br{ 70 }_{3}$ & $\br{ 10 }_{3}$ & $\br{ 20 }_{3}$ & 9 \\ \hline 
$\br{ 5 }_{1}$ & $\br{ 65 }_{1}$ & $\br{ 35 }_{3}$ & $\br{ 5 }_{3}$ & $\br{ 25 }_{3}$ & $\br{ 85 }_{3}$ & 0 \\ \hline
$\br{ 1 }_{4}$ & $\br{ 13 }_{4}$ & $\br{ 3 }_{12}$ & $\br{ 19 }_{12}$ & $\br{ 17 }_{12}$ & $\br{ 7 }_{12}$ & 0 \\ \hline 
$\br{ 2 }_{4}$ & $\br{ 26 }_{4}$ & $\br{ 14 }_{12}$ & $\br{ 2 }_{12}$ & $\br{ 4 }_{12}$ & $\br{ 8 }_{12}$ & 0 \\ \hline   \hline 
 & $y_1$ & $3y_2$ & $3y_3$ & $3y_4$ & $3y_5$ &  
\end{tabular}
\caption{Orbit information for $CW(130,81)$}
\label{tab:130}
\end{center}
\end{table}

\end{proof}

\bigskip


\begin{proposition} A $CW(143,81)$ does not exist.  \end{proposition}

\bigskip

\begin{proof}
Suppose a $CW(143,81)$ exists; 
$|P|=45,|N|=36$ and 3 is its multiplier.  Note
that $\Znum_{143}=\Znum_{11}\times\Znum_{13}$.  The $\Znum_{11}$ orbits
under the multiplier action $x\rightarrow3x$ are: \[ \{0\} \] \[
\{1,3,9,5,4\} \] \[ \{2,6,7,10,8\} \]

The $\Znum_{13}$ orbits under the multiplier action $x\rightarrow3x$
are: \[ \{0\} \] \[ \{1,3,9\} \] \[ \{2,6,5\} \] \[
\{4,12,10\} \] \[ \{7, 8,11\} \]
Table~\ref{tab:143} gives the orbit information.

\begin{table}
\begin{center}
\begin{tabular}{c||c|c|c|c|c||c}
\multicolumn{4}{c}{\ \ \ \ \ \ \ \ \ \ $\Znum_{13}$}\\
$\Znum_{11}$ & $\br{ 0 }_{1}$ & $\br{ 1 }_{3}$ & $\br{ 2 }_{3}$ & $\br{ 4 }_{3}$ & $\br{ 7 }_{3}$ &  \\ \hline  \hline 
$\br{ 0 }_{1}$ & $\br{ 0 }_{1}$ & $\br{ 22 }_{3}$ & $\br{ 44 }_{3}$ & $\br{ 77 }_{3}$ & $\br{ 11 }_{3}$ & $x_1$ \\ \hline 
$\br{ 1 }_{5}$ & $\br{ 26 }_{5}$ & $\br{ 1 }_{15}$ & $\br{ 5 }_{15}$ & $\br{ 4 }_{15}$ & $\br{ 20 }_{15}$ & $5x_2$ \\ \hline 
$\br{ 2 }_{5}$ & $\br{ 13 }_{5}$ & $\br{ 29 }_{15}$ & $\br{ 2 }_{15}$ & $\br{ 10 }_{15}$ & $\br{ 7 }_{15}$ & $5x_3$ \\ \hline   \hline 
 & $x$ & $3y_0$ & $3y_1$ & $3y_2$ & $3y_3$ &  
\end{tabular}
\caption{Orbit information for $CW(143,81)$}
\label{tab:143}
\end{center}
\end{table}


Unfortunately, 3 is not self-conjugate modulo 11 or 13, so we need to
work a bit harder.

Applying Lemma~\ref{lem:subgp}, we get equations
\begin{eqnarray*}
y_1 + 3(y_2 + y_3 + y_4 + y_5) & = & 9 \\  
y_1^2 + 3(y_2^2 + y_3^2 + y_4^2 + y_5^2) & = & 81 \\  
\end{eqnarray*}
for $\Znum_{13}$, and
\begin{eqnarray*}
x_1 + 5(x_2+x_3) & = & 9 \\  
x_1^2 + 5(x_2^2 + x_3^2) & = & 81 \\  
\end{eqnarray*}
for $\Znum_{11}$.

For $\Znum_{11}$ the integer solutions are
$(9,0,0)$, $(4,3,-2)$, and $(-6,3,0)$ (together with
swapping the second and third coordinates).

The first one is impossible, since it forces three size-3 orbits in
the first row to be in $P$, leaving 36 remaining elements, but the
orbits in the other rows all have size a multiple of 5.  Similarly for
the second solution, the first row sum being 4 means that we either
have orbits in the first row contributing 4 (the size-1 and a size-3
orbit in $P$, and none in $N$) or 7 (the size-1 and two size-3 orbits
in $P$, and the remaining orbit in $N$), but again the number of
remaining elements is not a multiple of 5.

Finally for $(x_1,x_2,x_3) = (-6,3,0)$, the first row is forced to have two
size-3 orbits in $N$ and none in $P$, or three in $N$ and one in $P$.
The latter is impossible, since it would leave 42 elements in $P$ and
27 in $N$ to be covered by the orbits in the other rows, all of which
have size a multiple of 5.  The former leaves 45 elements of $P$
and 30 of $N$ for the other rows, and so one row must have two size-15
orbits in $P$ and one in $N$, while the other has one of each.  The
size-5 orbits cannot be used, so $x$ must be $0$.

There are three solutions to the $\Znum_{13}$ equations with $y_1=0$:
$(0,3,3,0,-3)$,
$(0,4,1,1,-3)$, and
$(0,5,0,-1,-1)$,
as well as permutations of the last four coordinates.

The first two may be quickly eliminated; in both cases we need a
column sum equal to $-9$, and it is not possible to achieve this.
However, the third solution can be satisfied.  Table~\ref{tab:143b}
shows a selection satisfying all the equations, although the orbits do
not form a $CW(143,81)$.

\begin{table}
\begin{center}
\begin{tabular}{c||c|c|c|c|c||c}
\multicolumn{4}{c}{\ \ \ \ \ \ \ \ \ \ $\Znum_{13}$}\\
$\Znum_{11}$ & $\br{ 0 }_{1}$ & $\br{ 1 }_{3}$ & $\br{ 2 }_{3}$ & $\br{ 4 }_{3}$ & $\br{ 7 }_{3}$ &  \\ \hline  \hline 
$\br{ 0 }_{1}$ & \phantom{$\br{ 0 }_{1}$} & \phantom{$\br{ 22 }_{3}$} & {$\br{ 44 }_{3}$} & {$\br{ 77 }_{3}$} & \phantom{$\br{ 11 }_{3}$} & $-6$ \\ \hline 
$\br{ 1 }_{5}$ & \phantom{$\br{ 26 }_{5}$} & {$\mathbf {\br{ 1 }_{15}}$} & {$\br{ 5 }_{15}$} & {$\mathbf{\br{ 4 }_{15}}$} & \phantom{$\br{ 20 }_{15}$} & $15$ \\ \hline 
$\br{ 2 }_{5}$ & \phantom{$\br{ 13 }_{5}$} & \phantom{$\br{ 29 }_{15}$} & {$\mathbf{\br{ 2 }_{15}}$} & {$\br{ 10 }_{15}$} & \phantom{$\br{ 7 }_{15}$} & $0$ \\ \hline   \hline 
 & $0$ & $15$ & $-3$ & $-3$ & $0$ &  
\end{tabular}
\caption{A choice of orbits satisfying the equations ({\em not} a
  $CW(143,81)$).   {$P$ orbits are in {\bf bold}}}
\label{tab:143b}
\end{center}
\end{table}

To finish the proof, consider the orbits in $\Znum_{13}$.  There are
six ways to pick two of the four columns with column sum $-3$, and then
two ways to pick which of the other columns has sum 15.  For each of
these 12 choices, we can check that (\ref{eq:cw}) is not satisfied.
For example, the choices in Table~\ref{tab:143b} give
$$
A(X) \equiv 5 X^{\br{ 1 }} - X^{\br{ 2 }} - X^{\br{ 4 }} \pmod{X^{13}-1},
$$
where $X^{\br{a}} = \sum_{b \in \br{a}} X^b$ for the orbit $\br{a}$ in
$\Znum_{13}$,
and we find
$$
A(X) A(X^{-1}) \equiv 81 + 18 \left(X^{\br{ 1 }} + X^{\br{ 2 }} -
X^{\br{ 4 }} + X^{\br{ 7 }} \right)  \pmod{X^{13}-1} \neq 81.
$$

\end{proof}

Finally, we have:

\begin{proposition} A $CW(143,36)$ does not exist.  \end{proposition}

\begin{proof}

Since $k$ is not a prime power, Theorem~\ref{thm:mult} does not
apply.  However, a more general multiplier theorem (\cite{arasuxiang},
Theorem 2.4) shows that 3 is still a multiplier, and so the orbit
information is exactly the same as in Table~\ref{tab:143}.

The table is the same, but the equations are
\begin{eqnarray*}
y_1 + 3(y_2 + y_3 + y_4 + y_5) & = & 6 \\  
y_1^2 + 3(y_2^2 + y_3^2 + y_4^2 + y_5^2) & = & 36 \\  
\end{eqnarray*}
for $\Znum_{13}$, and
\begin{eqnarray*}
x_1 + 5(x_2+x_3) & = & 6 \\  
x_1^2 + 5(x_2^2 + x_3^2) & = & 36 \\  
\end{eqnarray*}
for $\Znum_{11}$.

The solutions to the
$\Znum_{11}$ equations are $(6,0,0)$,
$(-4,2,0)$ and $(-4,0,2)$.  
The solutions to the
$\Znum_{13}$ equations are $(6,0,0,0,0)$ and
$(0,2,2,-2,0)$ and permutations of the last four coordinates.

But none of the $\Znum_{11}$ and
$\Znum_{13}$  are compatible; $(6,0,0)$ would force two or more of the
size-3 orbits in the first row to be in $P$.  The corresponding
columns would then have weight $3 \pmod {15}$, which does not fit with
any of the $\Znum_{13}$ solutions.  Similarly, $(-4,2,0)$ or
$(-4,0,2)$ would force the
$\br{0}_1$ orbit to be in $N$, so that the first column would have
weight $4 \pmod 5$, which is not compatible with
any of the $\Znum_{13}$ solutions.
\end{proof}

\bigskip






\section{Contracted Circulant Weighing Matrices}

For most of the remaining open cases in Strassler's table, we do not
have any multipliers from Theorem~\ref{thm:mult}, either because $k$
is composite or not relatively prime to $n$.


The following theorem,
due to McFarland \cite{mcfarland}, will sometimes allow us to obtain
multipliers in these cases:

\begin{theorem}\label{thm:mcf}
  Let $M$ be an $ICW_d(m,k)$ with $\gcd(m,k)=1$. Let $k$ have prime
  factorization $p_1^{e_1} \cdots p_s^{e_s}$.  If $t$ is an
  integer for which there are $f_i$ for $i=1,2,\ldots,s$ with
  \begin{equation}\label{eq:mcf}
t \equiv p_i^{f_i} \pmod m,
  \end{equation}
then $t$ is a multiplier of $M$.
\end{theorem}

Thus  for $A$ a putative $CW(n,k)$, we may apply this theorem to $A^{(d)}$
for $d = \gcd(n,k)$.
If such a $t$ exists, we will call it a $d$-multiplier for $A$
if we can find a $t$
satisfying (\ref{eq:mcf}), and we may apply the methods of the previous
section.

\begin{proposition} \label{prop:132-81}
A $CW(132,81)$ does not exist.  
\end{proposition}

\begin{proof}

By Theorem~\ref{thm:mcf}, $3$ is a multiplier for an $ICW_3(44,81)$.  
Table~\ref{tab:132} gives the orbit information.  Since this is an
ICW, any of the orbits may occur with a coefficient up to 3 in
absolute value.

Since 3 is self-conjugate modulo 4, by Theorem~\ref{thm:turyn} the row sums must be
$(9,0,0)$. This means that $\br{0}_1$ has a coefficient of $-1$, with
the other orbits in the first row having coefficients $(1,1)$,
$(2,0)$, or $(3,-1)$ in either order.

\begin{table}
\begin{center}
\begin{tabular}{c||c|c|c||c}
 \multicolumn{4}{c}{\ \ \ \ \ $\Znum_{11}$} \\
$\Znum_{4}$    & $\br{ 0 }_1$   & $\br{ 1 }_5$ & $\br{ 2 }_5$  &  \\ \hline \hline
$\br{ 0 }_1$  & $\br{ 0 }_1$ & $\br{ 4 }_5$ & $\br{ 8 }_5$ & $9$ \\ \hline
$\br{ 2 }_1$  & $\br{ 22 }_1$ & $\br{ 14 }_5$ & $\br{ 2 }_5$ & $0$ \\ \hline
$\br{ 1 }_2$  & $\br{ 11 }_2$ & $\br{ 1 }_{10}$ & $\br{ 7 }_{10}$  & $0$ \\ \hline \hline
 & $y_1$  & $5y_2$  & $5y_3$ & \\
\end{tabular}
\caption{Orbit information for $ICW_3(44,81)$}
\label{tab:132}
\end{center}
\end{table}

The solutions to the $\Znum_{11}$ equations are $(9,0,0)$,
$(4,3,-2)$, $(-6,0,3)$, and permutations of the last two columns.  But
since the second and third row sums are zero, and the other orbits
all have order $0 \pmod 5$, the coefficient of the orbits in those
rows in the first column must be zero.  None of the solutions has
first coefficient $-1$, so no $ICW_3(44,81)$ exists, and so no
$CW(132,81)$ exists.

\end{proof}







\section{Strassler's Table and Beyond}\label{sec:strassler}

There has been a large amount of work on entries in Strassler's table
in the past few years.  In particular, Tan \cite{tan} showed nonexistence for 19
cases, and gave an updated version of the table with 34 open cases remaining
($CW(126,64)$ and $CW(198,100)$ were listed as open in Tan's thesis,
although it was already known that they could be
constructed using Theorem~2.2 of \cite{ad99};
this was corrected in the published paper).
The seven cases resolved above leave
27 open cases.

Of the remaining cases, while the above methods do not yield
hand-checkable proofs, when there is a sufficiently large multiplier
group a computer exhaust becomes quite feasible.  
The search begins in the upper
right corner of Table~\ref{tab:generic}, and scans the boxes right to
left, doing each from from the bottom to the top.  For each orbit it
in turn skips it, adds it to $P$, and adds it to $N$, updating the row
and column sums and recursing.
For every possible set of row and column sums $({\bf r},{\bf c})$, we
call
${\tt Exhaust}(u,v,|\mcb_{uv}|,\emptyset,\emptyset,{\bf r},{\bf c})$.

\begin{algorithm}
\DontPrintSemicolon
\SetKwFunction{Fexhaust}{Exhaust}
\SetKw{KwTo}{downto}
\Fexhaust{$i$,$j$,$l$,$P$,$N$,{\bf r},{\bf c}}{\;
\If(\tcp*[f]{finished, test if we succeeded}){$i=j=l=0$}{  
    \If{${\bf r} = {\bf c} = {\bf 0}$ and \{$P,N\}$ forms a $CW(n,k)$}
    {report $\{P,N\}$}
    \Return\;
}
\If(\tcp*[f]{do orbits in this box}){$l>0$} {
  \Fexhaust{$i,j,l-1,P,N,{\bf r},{\bf c}$}\tcp*[f]{recurse without $l$th orbit}\;
  \Fexhaust{$i,j,l-1,P\cup \mcb_{i,j}^l,N,(r_1,\ldots
    r_i-|\mcb_{i,j}^l|,\ldots,r_u),(c_1,\ldots
    c_j-|\mcb_{i,j}^l|,\ldots,c_v)$}\tcp*[f]{$l$th orbit in $P$}\;
  \Fexhaust{$i,j,l-1,P,N\cup \mcb_{i,j}^l,(r_1,\ldots
    r_i+|\mcb_{i,j}^l|,\ldots,r_u),(c_1,\ldots
    c_j+|\mcb_{i,j}^l|,\ldots,c_v)$}\tcp*[f]{$l$th orbit in $N$}\;
  \ElseIf{$i>0$}{
    \Fexhaust{$i-1,j,|\mcb_{i-1,j}|,P,N,{\bf r},{\bf c}$}\tcp*[f]{do next  box in row}\;
    \ElseIf(\tcp*[f]{row done, stop if row sum nonzero}){$r_j=0$}{
      \Fexhaust{$u,j-1,|\mcb_{u,j-1}|,P,N,{\bf r},{\bf c}$}\tcp*[f]{start next row up}\;
    }
  }
}

}

    \caption{Exhaustive search for $CW(n,k)$}\label{alg:exhaust}
\end{algorithm}

We were able to
eliminate $CW(144,49)$, $CW(152,49)$, $CW(160,49)$, $CW(104,81)$, and
$CW(160,81)$.  The longest of these, $CW(144,49)$, had 27 solutions to
equations (\ref{eq:110a}) and (\ref{eq:110b}) modulo 9, and 252 solutions
modulo 16.  The computation took 15 days on a workstation, and
required testing $2.4$ billion putative circulant weighing matrices.

As stated, Algorithm~\ref{alg:exhaust} will find all $CW(n,k)$, or show that none
exist.  To find $ICW_m(n,k)$, the same algorithm works, allowing
up to $m$ copies of each orbit to be added to $P$ or $N$.
Table~\ref{tab:ContractedMults} gives open cases where we
can apply Theorem~\ref{thm:mcf} with a reasonably large $m$.  
Since there
is no $ICW_2(91,81)$, we have:

\begin{proposition} \label{prop:182-64}
A $CW(182,64)$ does not exist.
\end{proposition}

\begin{table}
\begin{center}
\begin{tabular}{|c|c|c|c|c|c|}
\hline
$n$ & $k$ & $m$ & $t$ & $|M|$ & \# $ICW_d(m,k)$\\ \hline
$105$ & $36$	&	$35$	&	$4$ & $6$ & $1$\\ 
$112$ & $36$	&	$7$	&	$2$ & $3$ & $2$\\
$117$ & $36$	&	$13$	&	$3$ & $3$ & $3$\\   
$140$ & $36$	&	$35$	&	$4$ & $6$ & $1$\\ 
$195$ & $36$	&	$65$	&	$16$ & $3$ & $4$\\   
$140$ & $64$	&	$35$	&	$2$ & $12$ & $3$\\   
$180$ & $64$	&	$45$	&	$2$ & $12$ & $1$\\   
$182$ & $64$	&	$91$	&	$2$ & $12$ & $0$\\
$196$ & $64$	&	$49$	&	$2$ & $21$ & $3$\\
$132$ & $81$	&	$44$	&	$3$ & $10$ & $0$\\   
$156$ & $81$	&	$52$	&	$3$ & $6$ & $100$\\
$195$ & $81$	&	$65$	&	$3$ & $12$ & $2$\\ 
$198$ & $81$	&	$22$	&	$3$ & $5$ & $13$\\
$156$ & $100$	&	$39$	&	$5$ & $4$ & $6$\\
$165$ & $100$	&	$33$	&	$4$ & $5$ & $8$\\   
$195$ & $100$	&	$39$	&	$5$ & $4$ & $6$\\
\hline
\end{tabular}
  \caption{$ICW$s for Open Cases}\label{tab:ContractedMults}
\end{center}
  \end{table}

For the other cases, there are $ICW$s that could potentially be lifted
to the corresponding CW. It is likely that further computations
could eliminate some of these, similar to how \cite{ArasuNabavi}
showed that there was no lift of an $ICW_2(77,36)$ to a $CW(154,36)$,
or of an $ICW_2(85,64)$ to a $CW(170,64)$.

\begin{table}
\begin{center}
\begin{tabular}{|c|c||c|c||c|c||c|c||c|c|}  \hline
$n$ & $k$  &	$n$ & $k$  &	$n$ & $k$  &	$n$ & $k$  &	$n$ & $k$ \\ \hline\hline
105	&	36	&	116	&	49	&	140	&	64	&	156	&	81	&	112	&	100	\\	
112	&	36	&	120	&	49	&	180	&	64	&	195	&	81	&	120	&	100	\\	
117	&	36	&	192	&	49	&	196	&	64	&	198	&	81	&	155	&	100	\\	
140	&	36	&		&		 &		&		&		&		&156	&	100	\\		
180	&	36	 & 		&		&              &             &	&		 &165	&	100	\\	
195	&	36	&              &             &              &             &              &             &182	&	100	\\	
	&		&              &             &              &             &              &             &195	&	100	\\\hline	
\end{tabular}
\caption{Remaining Open Cases with $n \leq 200$, $k \leq 100$}
\label{tab:strassler}
\end{center}
\end{table}

The remaining cases either have no known multipliers or a very small
multiplier group, so the methods of this paper will not work to
eliminate them.  Hopefully some new ideas will soon allow Strassler's
table to be fully settled, as Lander's table of difference set 
cases were twenty years ago \cite{iiams}.

As with Lander's table for difference sets, the parameters for
Strassler's table, $n \leq 200$ and $s \leq 10$, were a convenient
focus on approachable problems, not a hard limit never to be
exceeded.  
The code written for the above searches can handle larger numbers, so
we have started exploring further.  
The second author has set up an online database \cite{ljcr},
which contains a current version of the table, along with
known circulant weighing matrices 
for parameters in Strassler's table. 
It also has partial results for
for $n \leq 1000$ and $k \leq 19^2$. 
Out of the 15982 such parameters, 1175 have $CW$s, 12017 do not, and 2790 remain
open.

\section{Proper $CW(n,k)$}

Recall that a $CW(n,k)$ is called proper if it is not a multiple of
any smaller $CW$, i.e. its group ring representation $A(X)$ is not
equal to $B(X^d)$ for any $n=dm$ for $B(X)$ a $CW(m,k)$.
For example, 
$$A=X+X^2+X^3+X^6+X^9+X^{18}-X^4-X^{12}-X^{10}$$
is a proper $CW(26,3^2)$ (i.e. no $X^a A(X^b)$ for $b$ relatively
prime to $26$ has all its coefficients with a common factor),
while
$$A=X^2+X^8+X^{10}+X^{12}+X^{14}+X^{20}-1-X^{4}-X^{16}$$
is a multiple of the proper $CW(13,3^2)$
$$A=X+X^4+X^{5}+X^{6}+X^{7}+X^{10}-1-X^{2}-X^{8}.$$

Clearly it suffices to study proper $CW$s, and restricting our
attention to those lets us present the state of knowledge about
circulant weighing matrices in a form far more compact than
Strassler's table.  In this section we give the known results, which
almost entirely come from two constructions.

Leung and Schmidt \cite{ls2011} showed that when $k$ is an odd prime
power there are only a finite
number of proper $CW(n,k)$.
For which $k$ can we give a complete list of
proper $CW(n,k)$?  This has been solved for $k=4$ \cite{eh} and $k=9$
\cite{aalms08}.
For $k=25$ Leung and Ma
\cite{lm1} show that none exist with $n \equiv 0 \pmod 5$, and in a
2011 preprint  \cite{lm2} deal with the other cases, although this has
not appeared in print.

For $k=16$ this question was not completely answered. In
\cite{arasu_etal_2006}, it is shown that all proper $CW(n,16)$ have
either $n=21, 31,63$ or are of ``Type II'', meaning that they are
constructed using Theorem 2.3 of that paper:

\begin{theorem}\label{thm:2.3}
  If $B$ is a $CW(2n,k)$, and $C$ is a $CW(n,k)$.  If the supports of
  $B(X)$, $X^n B(X)$,$C(X^2)$, $X^n C(X^2)$ are pairwise disjoint,
  then 
$$(1-X^n) B(X) + (1+X^n) C(X^2)$$
is a $CW(2n,4k)$.
\end{theorem}

With this we can classify the proper $CW(n,16)$ of
even order:

\begin{theorem}
  The  proper $CW(n,16)$ have order 21,31,63, and $14m$ for all $m \geq 2$.
\end{theorem}

\begin{proof}
The odd orders were taken care of in \cite{arasu_etal_2006}.
Let $C=-1+X+X^2+X^4$ denote the $CW(7,4)$, and 
$$
A = (1-X^{7m})C(X^{2m})  + (1+X^{7m}) X C(X^m).
$$
The coefficients of $A$ are disjoint for $m>1$, so by
Theorem~\ref{thm:2.3} $A$ is a
$CW(14m,16)$.  Since the coefficients of $X^0$, $X^1$, $X^{2m}$ and
$X^{7m}$ are nonzero, no equivalent difference set has all terms
divisible by $2$, $7$ or $m$, so it is a proper one.

The only other way to construct a Type II CW would be to use one of
the $CW(2m,4)$. 
Proper ones are equivalent to $-1+X+X^{m} + X^{m+1}$
and so in Theorem~\ref{thm:2.3} the supports would not be disjoint.
Improper ones are either a multiple of $CW(7,4)$, or also have nonzero
coefficients for $X^0$ and $X^m$, and so also fail the requirements of
the theorem.
\end{proof}

For larger $k$ not much is known.
Table~\ref{tab:proper} gives a list of known proper $CW(n,k)$ for 
$k \leq 19^2$.  Aside from small cases, they all come from
Theorems~\ref{thm:kronecker} and \ref{thm:rds} below.

The Kronecker product  construction of Arasu and Seberry \cite{arasuseberry}
accounts for almost all of the proper $CW(n,s^2)$ for $s$
not prime, and all the infinite classes except for $CW(2m,2^2)$
\cite{eh} and $CW(48m,6^2)$ \cite{ss13}:

\begin{theorem}\label{thm:kronecker}
  If a proper $CW(n_1,k_1)$ and proper $CW(n_2,k_2)$ exist with $\gcd(n_1,n_2)=1$,
  then they may be used to construct a proper $CW(n_1 n_2,k_1 k_2)$
\end{theorem}

For $k$ a prime power, most $CW(n,k)$ come from relative difference sets.
A {\em $(m,n,k,\lambda)$ cyclic relative difference set (RDS)} $D$ is a
$k$-element subset of
$\Znum_{mn}$ such that
$$
D D^{-1} = k + \lambda (\Znum_{mn} - \Znum_n).
$$
See \cite{pott} for more information on relative difference sets.
It is well known (e.g. Theorem 2.1 of \cite{arasu_etal_2006}):

\begin{theorem}
  If a cyclic $(m,2n,k,\lambda$)-RDS exists, then there is a $CW(mn,k)$.
\end{theorem}

In \cite{crds2001} it is shown:
\begin{theorem}
For $q$ a prime power,
a cyclic   $\left( \frac{q^d-1}{q-1},n,q^{d-1},q^{d-2}(q-1)/n
\right)$-RDS exists if and only if $n$ is a divisor of $q-1$ when $q$
is odd or $d$ is even, and if and only if $n$ is a divisor of $2(q-1)$
if $q$ is even and $d$ is odd.
\end{theorem}

Taking $d=3$, we have:

\begin{theorem}\label{thm:rds}
Let $q$ be a prime power.  Then a proper $CW((q^3-1)/n,q^2)$ exists for all divisors $n$ of $(q-1)$
if $q$ is even, and all divisors $n>1$ of $(q-1)$ if $q$ is odd.
\end{theorem}

All the proper $CW(n,k)$ in Table~\ref{tab:proper} coming from this
theorem are in bold.  
This theorem shows that there are 
proper $CW(q^3-1,q^2)$ when $q$ is an even prime power.  They also exist
for $q=3$ and $5$, 
and it is tempting to
conjecture that this is true for all prime powers, but larger cases
are currently out of reach.


\begin{table}
\begin{center}
\begin{tabular}{|c|l|}  \hline
$k$ & Known Proper $CW(n,k)$  \\ \hline
$2^2$	&	\underline{$2m$}, {\bf 7}	\\
$3^2$	&	{\bf 13}, \underline{24}, \underline{26}	\\
$4^2$	&	$14m$, {\bf 21}, \underline{31}, {\bf 63}	\\
$5^2$	&	{\bf 31}, \underline{33}, {\bf 62}, \underline{71}, \underline{124}, \underline{142}	\\
$6^2$	&	$26m$, \underline{$48m$}, 91, 168 \\
$7^2$	&	{\bf 57}, \underline{87},  {\bf 114}, {\bf 171}	\\
$8^2$	&	$42m$, $62m$, {\bf 73}, \underline{127}, 217, {\bf 511}	\\
$9^2$	&	{\bf 91}, \underline{121},  {\bf 182}, {\bf 364}	\\
$10^2$	&	$62m$, $66m$, $142m$, 217, 231, 497, 994 \\
$11^2$	&	{\bf 133}, {\bf 665}	\\
$12^2$	&	$182m$, $336m$, 273, 403, 744	\\
$13^2$	&	{\bf 183}, {\bf 366}, {\bf 549}, {\bf 732}	\\
$14^2$	&	$114m$, $174m$, $342m$, 399, 609	\\
$15^2$	&	403, 429, 744, 806, 923 \\
$16^2$	&	$146m$, $254m$, $434m$, {\bf 273}, 511, 651, {\bf 819}, 868, 889	\\
$17^2$	&	{\bf 307}, {\bf 614}	\\
$18^2$	&	$182m$, $242m$, $624m$, 847	\\
$19^2$	&	{\bf 381}, {\bf 762}	\\ \hline
\end{tabular}
\caption{Known Proper $CW(n,k)$. Numbers in {\bf bold} come from
  Theorem~\ref{thm:rds}. Underlined entries are sporadic $CW$s that do not come from
Theorems~\ref{thm:kronecker} or \ref{thm:rds}. Entries $cm$ are for all $m$ such that $cm
  \geq k$. }
\label{tab:proper}
\end{center}
\end{table}




\bibliography{agz}

\begin{thebibliography}{10}

\bibitem{aalms08}
M.H. Ang, K.T. Arasu, S.L. Ma, and Y.~Strassler.
\newblock Study of proper circulant weighing matrices with weight $9$.
\newblock {\em Disc. Math.}, 308:2802--2809, 2008.

\bibitem{arasuhollon}
K.~T. Arasu and J.~R. Hollon.
\newblock Group developed weighing matrices.
\newblock {\em Australas. J. Combin.}, 55:205--233, 2013.

\bibitem{arasuseberry2}
K.~T. Arasu and J.~Seberry.
\newblock Circulant weighing designs.
\newblock {\em J. Combin. Des.}, 4:439--447, 1996.

\bibitem{arasuseberry}
K.~T. Arasu and J.~Seberry.
\newblock On circulant weighing matrices.
\newblock {\em Australas. J. Combin.}, 17:21--37, 1998.

\bibitem{arasuxiang}
K.~T. Arasu and Q.~Xiang.
\newblock Multiplier theorems.
\newblock {\em J. Combin. Des.}, 3:257--268, 1995.

\bibitem{ad99}
K.T. Arasu and J.F. Dillon.
\newblock Perfect ternary arrays.
\newblock In {\em Difference Sets, Sequences and their Correlation Properties},
  pages 1--15. Kluwer, 1999.

\bibitem{crds2001}
K.T. Arasu, J.F. Dillon, K.H. Leung, and S.L. Ma.
\newblock Cyclic relative difference sets with classical parameters.
\newblock {\em JCT A}, 94:118--126, 2002.

\bibitem{arasu_etal_2006}
K.T. Arasu, K.H. Leung, S.L. Ma, A.~Nabavi, and D.K.Ray-Chaudhuri.
\newblock Circulant weighing matrices of weight $2^{2t}$.
\newblock {\em Designs, Codes and Cryptography}, 41:111--123, 2006.

\bibitem{ArasuNabavi}
K.T. Arasu and A.~Nabavi.
\newblock Nonexistence of {CW}(154,36) and {CW}(170,64).
\newblock {\em Disc. Math.}, 311:769--779, 2011.

\bibitem{dk2015}
D.Z. Dokovi\'{c} and I.S. Kotsireas.
\newblock Compression of periodic complementary sequences and applications.
\newblock {\em Designs, Codes and Cryptography}, 74:365--377, 2015.

\bibitem{eh}
P.~Eades and R.M. Hain.
\newblock On circulant weighing matrices.
\newblock {\em Ars Combin.}, 2:265--284, 1976.

\bibitem{ljcr}
D.~M. Gordon.
\newblock La {J}olla {C}ombinatorics {R}epository.
\newblock \url{https://www.dmgordon.org/}, 2021.

\bibitem{iiams}
J.~Iiams.
\newblock Lander's tables are complete!
\newblock In {\em Difference Sets, Sequences and their Correlation Properties},
  pages 239--257. Kluwer, 1999.

\bibitem{jungnickel}
D.~Jungnickel.
\newblock On {L}ander's multiplier theorem for difference lists.
\newblock {\em J. Comb. Info. and Syst. Sci.}, 17:123--129, 1992.

\bibitem{lm2}
K.H. Leung and S.L. Ma.
\newblock Proper circulant weighing matrices of weight 25.
\newblock preprint, 2011.

\bibitem{lm1}
K.H. Leung and S.L. Ma.
\newblock Proper circulant weighing matrices of weight $p^{2}$.
\newblock {\em Designs, Codes and Cryptography}, 72:539--550, 2014.

\bibitem{ls2011}
K.H. Leung and B.~Schmidt.
\newblock Finiteness of circulant weighing matrices of fixed weight.
\newblock {\em JCT A}, 118:908--919, 2011.

\bibitem{mcfarland}
R.L. McFarland.
\newblock {\em On multipliers of abelian difference sets}.
\newblock PhD thesis, The Ohio State University, 1970.

\bibitem{pott}
A.~Pott.
\newblock {\em Finite Geometry and Character Theory}, volume 1601 of {\em
  Lecture Notes in Mathematics}.
\newblock Springer, 1995.

\bibitem{ss13}
B.~Schmidt and K.W. Smith.
\newblock Circulant weighing matrices whose order and weight are products of
  powers of 2 and 3.
\newblock {\em JCT A}, 120:275--287, 2013.

\bibitem{strassler}
Y.~Strassler.
\newblock {\em The Classification of Circulant Weighing Matrices of Weight 9}.
\newblock PhD thesis, Bar-Ilan University, 1997.

\bibitem{tan}
M.M. Tan.
\newblock Group invariant weighing matrices.
\newblock {\em Designs, Codes and Cryptography}, 86:2677--2702, 2018.

\end{thebibliography}
\bibliographystyle{plain}

\end{document}